\documentclass[12pt]{article}

\usepackage[margin=1in]{geometry} 
\usepackage{amsmath,amsthm,amssymb}
\usepackage{xcolor}
\usepackage{graphicx}
\usepackage{caption}
\usepackage{subcaption}
\usepackage[utf8]{inputenc}
\usepackage[T1]{fontenc}
\usepackage{mathtools}
\usepackage[thinc]{esdiff}
\usepackage{derivative}
\usepackage{blindtext}
\usepackage{hyperref}

\newtheorem{Theorem}{Theorem}[section]

\newtheorem{Corollary}[Theorem]{Corollary}
\newtheorem{Definition}[Theorem]{Definition}
\newtheorem{Lemma}[Theorem]{Lemma}
\newtheorem{Remark}[Theorem]{Remark}
\newtheorem{Example}[Theorem]{Example}
\newtheorem{Proposition}[Theorem]{Proposition}

\begin{document}


\title{On the existence of a balanced vertex in geodesic nets with three boundary vertices}
\author
{Duc Toan Nguyen \thanks{Texas Christian University, Fort Worth, TX, USA (duc.toan.nguyen@tcu.edu)} }
\date{}
\maketitle
\begin{abstract}
    Geodesic nets are types of graphs in Riemannian manifolds where each edge is a geodesic segment. One important object used in the construction of geodesic nets is a balanced vertex, where the sum of unit tangent vectors along adjacent edges is zero. We prove the existence of a balanced vertex of a triangle (with three unbalanced vertices) on a general two-dimensional Riemannian surface when all angles measure less than $2\pi/3$, if the length of the sides of the triangle is not too large. This property is a generalization for the existence of the Fermat point of a planar triangle.
\end{abstract}
\section{Introduction}

In \cite{MR4574417}, the authors introduce \emph{geodesic nets} as certain types of graphs embedded in a Riemannian manifold. A geodesic net on a Riemannian manifold $M$ consists of a finite set $V$ of vertices and a finite set $E$ of edges, which are non-constant geodesics between vertices. The set $V$ contains two types of vertices: \emph{balanced} and \emph{unbalanced}. Unbalanced vertices are vertices of degree 1. Each balanced point $v$ has two properties:
\begin{enumerate}
    \item The sum of all unit tangent vectors in $T_v M$ from $v$ towards all adjacent vertices is $0$.
    \item By convention, $v$ must have degree $\geq 3$.
\end{enumerate}
Also, the authors in \cite{MR4574417} require that no edge connects two unbalanced vertices. 
\\

A geodesic net is a generalization of a solution for the \emph{Steiner Tree} problem: given a finite set of points, find the graph spanning all vertices with minimum total length. In the case of three given points, the problem is the called Fermat-Torricelli problem. In that case, given a triangle $ABC$ on an Euclidean  plane such that all its angles measure less than $2\pi/3$, the Steiner tree contains a point, called a Fermat point, which is connected to $A,B,$ and $C$ and has the three surrounding angles measure $2\pi/3$. The Fermat point here satisfies the property of a balanced point of a geodesic net with three vertices.
\\

In 2021, Parsch, one the authors in \cite{MR4574417}, investigated a property of a geodesic net with three unbalanced vertices in \cite[Theorem 1.2]{MR4310934}: Each geodesic net with 3 unbalanced vertices on the plane endowed with a Riemannian metric of non-positive curvature has at most one balanced vertex. This theorem gives an upper bound for the number of balanced vertices in a geodesic net with three unbalanced vertices. In \cite[Theorem 3.1.2]{MR4574417}, Parsch proves that a given geodesic net with three unbalanced vertices on a non-positively curved Riemannian $\mathbb{R}^2$ has exactly one balanced vertex.
\\

In this paper, we investigate some conditions of three unbalanced vertices (or a triangle) for the existence of a balanced vertex given on a general Riemannian surface, not only surfaces with non-positive curvature. To be more specific, let $M$ be a two-dimensional surface and $\mathcal{U}$ be a small neighborhood of a point on $M$ that is homeomorphic to an open disk in $\mathbb{R}^2$. We assume that given two points $A$ and $B$ in $\mathcal{U}$, there exists only one geodesic from $A$ to $B$ that lies inside $\mathcal{U}$, and it is the shortest curve connecting those points. Here is our main theorem showing sufficient conditions for the existence of a balanced point on surfaces with Gaussian curvature bounded above by a constant.

\begin{Theorem} \label{thm:positive}
Let $M$ be a Riemannian surface such that its Gaussian curvature is bounded above by $1/R^2$, for $R>0$. Let triangle $ABC$ on $M$ be given such that its three angles measure less than $2\pi/3$. If the maximum geodesic distance of two points in the domain of the triangle $ABC$ is less than $R\pi/2$, then there exists a balanced point.
\end{Theorem}

This theorem generalizes the result of the existence of the Fermat point in a triangle with three angles that measure less than $2\pi/3$ on a plane. A direct corollary of Theorem \ref{thm:positive} is the following.

\begin{Corollary}
    On any neighborhood with compact closure on a complete Riemannian surface, let $R>0$ be a positive number such that the curvature on that neighborhood is bounded above by $1/R^2$. Thus, if a triangle $ABC$ contained in the neighborhood has three angles that measure less than $2\pi/3$ and the maximum geodesic distance of two points in the domain of the triangle is less than $R\pi/2$, then there exists a balanced point.
\end{Corollary}

\textbf{Paper structure:} We summarize the structure of the paper as follows:
\begin{itemize}
    \item In \textbf{Section \ref{section:2}}, we present the preliminary background, including some definitions, notations, lemmas, and theorems that are used in the proofs of some results later.
    \item In \textbf{Section \ref{section:3}}, we prove Theorem \ref{thm:positive}.
    \begin{itemize}
        \item In \textbf{Subsection \ref{subsec:1}}, we first develop some ``warm-up'' results and construct the main condition for the existence of a balance point based on the increase of an angle; see Proposition \ref{prop:increase} and Proposition \ref{prop:condition_increase}.
        \item In \textbf{Subsection \ref{subsection:3.2}}, 
        we prove Proposition \ref{thm:non_positive}, which is the  non-positive-curvature case of the main theorem. This proposition has a simple and direct proof from the ``warm-up'' results above.
        \item In \textbf{Subsection \ref{subsection:3.3}}, we prove Proposition \ref{thm:sphere} in the case of round spheres and Theorem \ref{thm:positive} in the case of surfaces with curvature bounded from above. At the end of this subsection, we give one example of a triangle on a round sphere that has no balanced point and one example showing that our condition is not sharp, meaning that there exist triangles on positive-curvature surfaces that have angles greater than $2\pi/3$ and yet have a Fermat point.
    \end{itemize}
\end{itemize}

\textbf{Related Results:} The definition of a \emph{geodesic net} was first introduced in \cite{MR1343696}, where Hass and Morgan showed the existence of some specific types of geodesic nets on a 2-sphere. In \cite{MR1618690}, Heppes extended the previous results by showing that there exists a geodesic net with vertices of degree 3 or 4 partitioning the round 2-sphere into $n$ regions for any natural number $n$. These two papers were inspired by the famous result stated by Poincaré: There exists a simple geodesic on any positively curved 2-sphere. This result was proved in \cite{MR683167, MR1417428}. In \cite{MR44760}, Lyusternik and Fet proved their famous Lyusternik–Fet theorem, which states that there exists a closed geodesic on every compact Riemannian manifold. Another related question is whether every closed manifold has infinitely many periodic geodesics, which is investigated in \cite{MR246329,MR978076, MR1209957, MR2326944}. In \cite{MR4574417}, Nabutovsky and Parsch survey a wide range of recent problems and results related to geodesic nets, which builds a first step for researchers interested in working with this topic. In 2015, Memarian \cite{MR3418465} investigated the problem of the maximum number of balanced vertices in a critical graph, i.e. a planar geodesic net, before the work by Parsch \cite{MR4310934}. The connection between minimal networks and geodesic nets on different manifolds is also studied in \cite{MR3616371, MR4079990}. Recently, in \cite{MR4595313}, Chambers et. al. showed that a geodesic flower, which is a finite collection of geodesic loops based at point $p$ that satisfies the balancing condition, is a stationary geodesic net. In \cite{MR4688379}, Liokumovich and Staffa proved that for a generic Riemannian metric on a closed smooth manifold, the union of the images of all stationary geodesic nets forms a dense set. Following up on Poincaré's question, Dey \cite{dey2023existence} showed the existence of closed geodesics on certain non-compact Riemannian manifolds.\\

\textbf{Acknowledgement:} The author would like to thank Professor Ken Richardson for his dedicated mentorship and valuable discussions. 

\section{Preliminary Background} \label{section:2}
Let us introduce some notation and definitions. First, we use $\gamma_{A,B}$ to denote the unique geodesic from $A$ to $B$ in $\mathcal{U}$ (assuming the existence and uniqueness of the geodesic). We let $d(A,B)$ be the geodesic distance between $A$ and $B$; equivalently, it is the length of the shortest curve connecting the two points in $M$. Next, we define the angle between two geodesics.

\begin{Definition} Let $A,B,C$ be three points on $\mathcal{U}$. We define the \textbf{angle} between two geodesics $\gamma_{X,A},\gamma_{X,B}$ at $X$, denoted by $\angle AXB$, as the smallest angle (among two possible angles) between two tangent vectors of the geodesics $\gamma_{X,A}$ and $\gamma_{X,B}$ at $X$ on the tangent plane $T_X M$. Also, we denote $m (\angle AXB)$ as the \textbf{measure} of $\angle AXB$.
\end{Definition}

\begin{figure}[h!]
    \centering
    \begin{subfigure}{.45\textwidth}
  \centering
  \includegraphics[width=.8\linewidth]{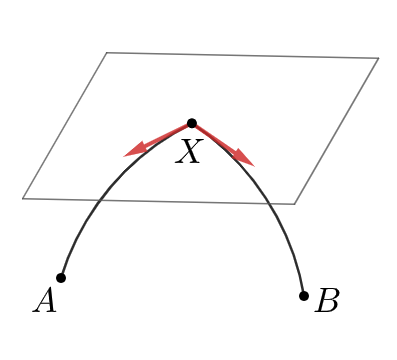}
  \caption{Angle $\angle AXB$}
  \label{fig:fig1a}
\end{subfigure}%
\begin{subfigure}{.45\textwidth}
  \centering
  \includegraphics[width=.8\linewidth]{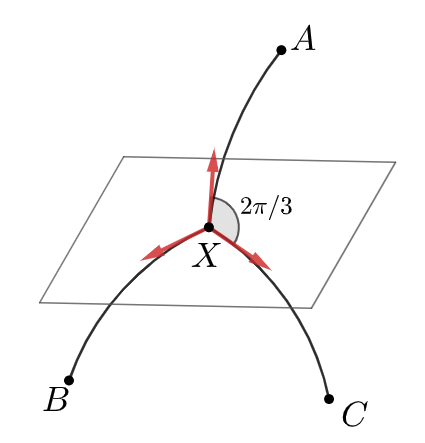}
  \caption{Balanced vertex}
  \label{fig:fig1b}
\end{subfigure}
    \caption{Angles}
    \label{fig:angle}
\end{figure}

\begin{Remark}
    With this definition, a \textbf{balanced vertex} $X$ corresponding to three vertices $A,B,C$ is a point lying inside the triangle $ABC$ (created by geodesics between pairs of vertices) such that
    $$m(\angle AXB) = m(\angle BXC) = m(\angle CXA) = \dfrac{2\pi}{3}$$
\end{Remark}

Now, we recall the Gauss-Bonnet Theorem for the case of a triangle and the Lebesgue number lemma. These facts are important for the main proofs below.

\begin{Theorem}[Gauss-Bonnet \cite{Kobayashi2019}] Given a smooth complete two-dimensional surface $M$,  let $A,B,C$ be three points on $M$. Connect points $A$ and $B$ with a directed smooth curve $\alpha_{A,B}$. Let the curves $\alpha_{B,C}$ and $\alpha_{C,A}$ be similarly constructed so that the three curves have empty intersections. We designate a specific domain bounded by three curves as $\mathcal{A}$. Denote the oriented boundary by $\partial \mathcal{A}$. Let $\iota_A, \iota_B, \iota_C$ be three interior angles at three vertices of the triangle. Then, we have
$$\int_{\mathcal{A}} K \, d\mathcal{A} + \int_{\partial \mathcal{A}} \kappa_g ds = \iota_A + \iota_B + \iota_C - \pi.$$
\end{Theorem}

\begin{Corollary}
    If $\alpha_{A,B}, \alpha_{B,C}, \alpha_{C,A}$ are three geodesics, then we obtain
    $$\int_{\mathcal{A}} K \, d\mathcal{A}  = \iota_A + \iota_B + \iota_C - \pi.$$
\end{Corollary}

\begin{Definition}
    Let $\mathcal{C}$ be an open cover of a metric space $X$. A \textbf{Lebesgue number} of the cover $\mathcal{C}$ is a positive real number $\delta$ such that for every subset $S$ of $X$ with diameter less than $\delta$, we have 
    $$S \subseteq C, \text{ for some $C \in \mathcal{C}$}.$$
\end{Definition}

\begin{Theorem}[The Lebesgue number lemma \cite{MR464128}] Every open cover of a compact metric space $(X,d)$ has a Lebesgue number.
\end{Theorem}

We will utilize Jacobi fields, which are fundamental objects in Riemannian geometry describing families of geodesics.

\begin{Definition} Let $\gamma:[0,a] \to M$ be a geodesic in a manifold $M$. A vector field $J$ along $\gamma$ is called a \textbf{Jacobi field} if it satisfies the Jacobi equation for all $t \in [0,a]$ 
\begin{align*}
    \frac{D^2J}{dt^2} + R(\gamma'(t),J(t))\gamma'(t) = 0,
\end{align*}
where $\dfrac{DJ}{dt}$ is the covariant derivative of the vector field $J$ along $\gamma$, and $R$ is the Riemann curvature operator of $M$. With the convention from \cite{do1992riemannian},
$$R(X,Y)Z = \nabla_Y \nabla_X Z - \nabla_X \nabla_Y Z + \nabla_{[X,Y]}Z,$$
for $X,Y,Z$ being any three vector fields and $\nabla$ being the Levi-Civita connection.
\end{Definition}

\begin{Lemma}
  In the case of a surface with Gaussian curvature $K$, if $J(t)$ is orthogonal to $\gamma'(t)$ for all $t$, then we can rewrite the Jacobi Equation as
\begin{align}\label{eqn:jacobi}
    \frac{D^2J}{dt^2} + KJ = 0.
\end{align}  
\end{Lemma}
This follows from \textbf{Example 2.3, Chapter 5} in \cite{do1992riemannian}. Now, denote $J'(t) := \frac{DJ}{dt}(t)$. From \textbf{Chapter 5} of \cite{do1992riemannian}, we have some properties for Jacobi Fields. 

\begin{Lemma}
    A Jacobi field is determined by its initial conditions $J(0)$ and $J'(0)$.
\end{Lemma}

\begin{Remark}
    If $J(0) = 0$ and $\lVert J'(0) \rVert = 0$, then $J \equiv 0$.
\end{Remark}

\begin{Lemma} \label{prop: jacobi_property}
    Let $J$ be a Jacobi field along the geodesic $\gamma:[0,a] \to M$. Then,
    $$\langle J(t), \gamma'(t) \rangle = \langle J'(0),\gamma'(0) \rangle t + \langle J(0),\gamma'(0) \rangle,\quad \text{ for } t \in [0,a].$$
\end{Lemma}

\begin{Corollary}
    If we have $J(0)=0$ and $\langle J'(0), \gamma'(0) \rangle = 0$, then we have $\langle J(t), \gamma'(t) \rangle = 0$ for all $t \in [0,a]$, i.e. the Jacobi fields are perpendicular to the tangent vector $\gamma'(t)$ along the geodesic $\gamma$.
\end{Corollary}

\begin{Proposition}
    Let $\gamma:[0,a] \to M$ be a geodesic in $M$ with $\dim M = n$, and let $\mathcal{J}^{\perp}$ be the space of Jacobi fields with $J(0) = 0$ and $J'(0) \perp \gamma'(0)$. Let $\lbrace J_1(t), \cdots, J_{n-1}(t) \rbrace$ be a basis of $\mathcal{J}^{\perp}$. If $\gamma(t)$, $t\in (0,a]$, is not conjugate to $\gamma(0)$ (i.e. $J(t) \neq J(0) = 0$), then $\lbrace J_1(t), \cdots, J_{n-1}(t) \rbrace$ is a basis for the orthogonal complement $\textnormal{span}\lbrace \gamma'(t)\rbrace^\perp \subset T_{\gamma(t)}M$.
\end{Proposition}

Finally, we introduce the \textbf{Index Lemma}, a fundamental lemma for proving comparison theorems in differential geometry.

\begin{Lemma}[The Index Lemma \cite{do1992riemannian}] \label{lemma:index} Let $\gamma: [0,a] \to M$ be a geodesic without conjugate points in the interval $(0,a]$. Let $J$ be a Jacobi field along $\gamma$, with $\langle J, \gamma' \rangle = 0$, and let $V$ be a piecewise differentiable vector field along $\gamma$, with $\langle V, \gamma' \rangle = 0$. Suppose that $J(0) = V(0) = 0$ and that $J(t_0) = V(t_0)$, for $t_0 \in (0,a]$. Then,
$$I_{t_0}(J,J) \leq I_{t_0}(V,V),$$
where
$$I_{t_0}(V,V) := \int_0^{t_0} (\langle V',V' \rangle - \langle R(\gamma',V)\gamma',V \rangle)dt.$$
\end{Lemma}

\begin{Remark}
    In the case of a surface with Gauss curvature $K$, we have
    \begin{align*}
        I_{t_0}(V,V) &= \int_0^{t_0} (\langle V',V' \rangle - \langle KV,V \rangle)dt \\
        &=\int_0^{t_0} (\langle V',V' \rangle - K \langle V,V \rangle)dt
    \end{align*}
\end{Remark}

\section{Existence of a balanced vertex in geodesic nets with three boundary vertices} \label{section:3}
On the plane, we recall the existence of a unique balanced vertex, as the Fermat point, of a triangle $ABC$ when the measures all three angles $\angle BAC$, $\angle ACB$, $\angle CBA$ of the triangle are less than $2\pi/3$. In this section, we investigate the existence of a balanced vertex corresponding to three points $A,B,C$ on neighborhood $\mathcal{U}$ on a general surface where all angles of triangle $ABC$ measure less than $2\pi/3$.

\subsection{General conditions for the existence of a balanced vertex} \label{subsec:1}
Using the notation of the previous section, we first prove the following.

\begin{Proposition} \label{prop:exist_120_AX}
    Let $X$ be a point lying on the geodesic $\gamma_{B,C}$ such that $X$ is different from $B$ and $C$. Then, there exists a point $Y$ between $A$ and $X$ (exclusively) on $\gamma_{A,X}$ such that $m(\angle BYC) = \dfrac{2\pi}{3}$ (Figure \ref{fig:exist_Y_X}).
\end{Proposition}

\begin{figure}[h!]
    \centering
    \includegraphics[width = .45\textwidth]{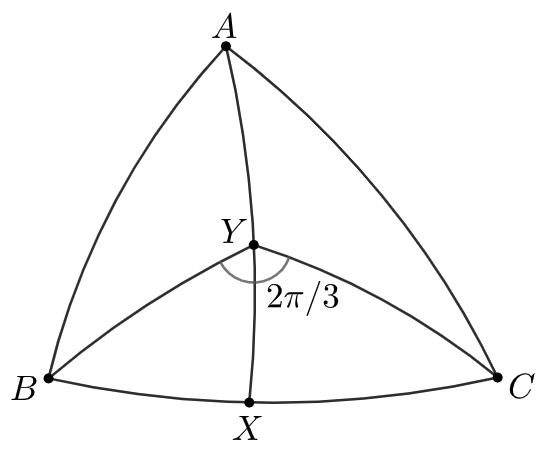}
    \caption{Existence of $Y$ with $m(\angle BYC)=\dfrac{2\pi}{3}$ for every $X$}
    \label{fig:exist_Y_X}
\end{figure}

\begin{proof}
    Since $X$ is on $\gamma_{B,C}$ and different from $B$ and $C$, then $m(\angle BXC) = \pi$. Let $Y$ be a moving point on $\gamma_{A,X}$ from $A$ to $X$ (inclusively). When $Y \equiv A$, we have $m(\angle BAC) <\dfrac{2\pi}{3}$; and when $Y \equiv X$, we have $m(\angle BXC) = \pi$. Also, since $M$ is a smooth surface, then $Y$ is moving continuously on $\gamma_{A,X}$ and $m(\angle BYC)$ changes continuously as $Y$ moves. By the Intermediate Value Theorem, there exists a point $Y'$ on $\gamma_{A,X}$ such that $m(\angle BY'C )= \dfrac{2\pi}{3}$.
\end{proof}

From this proposition, for each $X$ on $\gamma_{B,C}$, we denote $Y_X$ as the point that is closest to $A$ and $m(\angle BY_XC) = 2\pi/3$. Note that the existence of the unique point $Y_X$ closest to $A$ can be shown by the continuity of the angle function (i.e. the inverse image of a closed set is closed). Next, we propose a proposition for the existence of a balanced point based on the continuity of $Y_X$ as $X$ moves.

\begin{Proposition} \label{prop:increase}
    Let $\triangle ABC$ be a triangle on a neighborhood $\, \mathcal{U}$ of $M$ such that all angles measure less than $2\pi/3$. Let $X$ be a point lying on the geodesic $\gamma_{B,C}$ such that $X$ is different from $B$ and $C$. Let $Y_X$ be the closest point to $A$ such that $m(\angle BY_XC) = 2\pi/3$. Assume that the point $Y_X$ continuously varies inside $\triangle ABC$ as a function of $X$, as $X$ varies on $\gamma_{B,C}$. Then, there exists a point $T$ (balanced point) such that
    $$m(\angle ATB) = m (\angle BTC) = m(\angle CTA) = \dfrac{2\pi}{3}.$$
\end{Proposition}

\begin{figure}[h!]
    \centering
    \includegraphics[width=.55\textwidth]{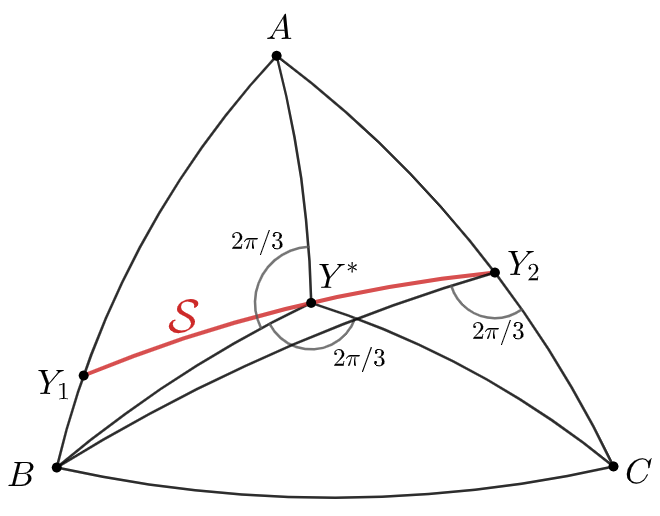}
    \caption{Existence of balanced point}
    \label{fig:main_theorem}
\end{figure}

\begin{proof}
    As $X$ moves continuously towards $B$, since $M$ is a smooth surface, then $\gamma_{A,X}$ also moves continuously towards $\gamma_{A,B}$. Then, since $Y_X$ varies continuously inside $\triangle ABC$, then $Y_X$ converges to a point $Y_1$ on $\gamma_{A,B}$ (Figure \ref{fig:main_theorem}). Also, since $Y_X$ is always between $A$ and $X$ (exclusively), then $Y_1$ is between $A$ and $B$ (inclusively). Similarly, as $X$ moves continuously towards $C$, $Y_X$ moves continuously towards $Y_2$ between $A$ and $C$ (inclusively) on $\gamma_{A,C}$. 

    From that, we have $Y_X$ moves on a continuous curve from $Y_1$ to $Y_2$ (exclusively), denoted by $\mathcal{S}$. Then, we will prove that there exists a balanced point $Y^{\ast}$ on $\mathcal{S}$.   
    \\
    
    First, we will show that $Y_2 \neq A$ and $m(\angle AY_2B) < \dfrac{2\pi}{3}$.
    \begin{itemize}
        \item \textbf{Case 1:} $Y_2 \neq C$. As $Y_X$ moves continuously on $M$ towards $Y_2$, the measure $m(\angle BY_XC)$ changes continuously. Then, $$m(\angle BY_2C) = \lim_{X \to C} m(\angle BY_XC) = \lim \dfrac{2\pi}{3} = \dfrac{2\pi}{3}.$$ Since $m(\angle BAC)<\dfrac{2\pi}{3}$, then $Y_2 \neq A$. Also, 
        $$m(\angle(AY_2B)) = \pi - m(\angle BY_2C) = \pi - \dfrac{2\pi}{3} = \dfrac{\pi}{3} < \dfrac{2\pi}{3}.$$
        \item \textbf{Case 2:} $Y_2 \equiv C$. Then, $Y_2 \equiv C \neq A$. Moreover, 
        $$m(\angle(AY_2B)) = m(\angle ACB) < \dfrac{2\pi}{3}.$$
    \end{itemize} 
    Thus, in both cases, we have $Y_2 \neq A$ and $m(\angle(AY_2B)) < \dfrac{2\pi}{3}$.
    \\
    
    Next, we will show that there exists a point $Y$ on $\mathcal{S}$ such that $m(\angle AYB) > 2\pi/3$. From that, by the Intermediate Value Theorem, there exists a balanced point on $\mathcal{S}$. Indeed, consider two cases:
    \begin{itemize}
        \item If $Y_1 \neq B$, then $m(\angle AY_1B) = \pi > \dfrac{2\pi}{3}$. We have $Y_X$ moves continuously on $\mathcal{S}$ from $Y_1$ to $Y_2$. In addition, $m(\angle AY_1B) > \dfrac{2\pi}{3}$, $m(\angle AY_2B) < \dfrac{2\pi}{3}$, and the measure $m(\angle AY_XB)$ is continuous (due to the smoothness of the manifold). By the Intermediate Value Theorem, there exists $Y^\ast$ between $Y_1$ and $Y_2$ on $\mathcal{S}$ such that $m(\angle AY^\ast B) = \dfrac{2\pi}{3}$. Thus, we have 
    $$m(\angle AY^\ast B) = m(\angle BY^\ast C) = m(\angle CY^\ast A) = \dfrac{2\pi}{3},$$
    whence $Y^*$ is a balanced point.
        \begin{figure}[h!]
            \centering
            \includegraphics[width = .55\textwidth]{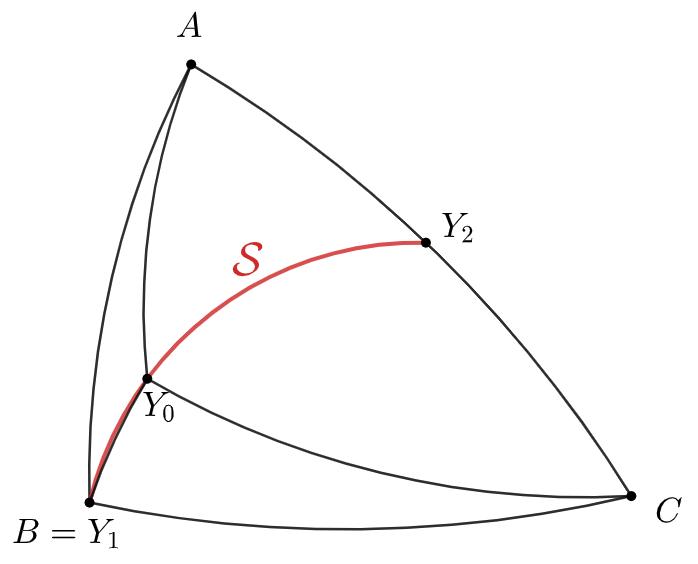}
            \caption{Existence of balanced point in case $Y_1 \equiv B$}
            \label{fig:main_theorem_Y_1_B}
        \end{figure}
        \item If $Y_1 \equiv B$ (Figure \ref{fig:main_theorem_Y_1_B}). We will show that the measure of the angle between the curve $\mathcal{S}$ and $\gamma_{B,C}$ at $B$ is $\pi/3$. As $Y_X$ approaches $B$ continuously, we have $\gamma_{Y_X,C}$ approaches $\gamma_{B,C}$, and they are nearly parallel. Let $\epsilon>0$. Then, there exists $Y_X$ on $\mathcal{S}$ such that $$m(\angle BY_XC) + m(\angle CBY_X) < \pi + \epsilon,$$ which means $$m(\angle CBY_X) = \pi + \epsilon - \dfrac{2\pi}{3} = \dfrac{\pi}{3} + \epsilon.$$ Thus, as $\epsilon \to 0$, we have the measure of the angle between curve $\mathcal{S}$ and $\gamma_{B,C}$ at $B$ is $\pi/3$. 
        
        Moreover, we have $m(\angle ABC) < 2\pi/3$. Denote $\Delta = 2\pi/3 - m(\angle ABC)>0$. When $Y_X$ tends to $Y_1 \equiv B$, we have $m(\angle Y_XBC)$ tends to $\pi/3$ and $(m(\angle ABY_X) + m(\angle AY_XB))$ tends to $\pi$. Let $\varepsilon>0$ such that $\varepsilon< \Delta$. Then, there exists $Y_0$  (Figure \ref{fig:main_theorem_Y_1_B}) on $\mathcal{S}$ such that
        \begin{itemize}
            \item[+)] $|m(\angle Y_0BC) - \pi/3| < \varepsilon/2$
            \item[+)] $|m(\angle ABY_0) + m(\angle AY_0B) - \pi|<\varepsilon/2$
        \end{itemize}
        Then,
        \begin{itemize}
            \item[+)] $\pi/3 - \varepsilon/2 < m(\angle Y_0BC)<\pi/3+\varepsilon/2$.
            \item[+)] $\pi - \varepsilon/2 < m(\angle ABY_0) + m(\angle AY_0B) < \pi + \varepsilon/2$.
        \end{itemize}
        Also, we have
        \begin{align*}
            m(\angle ABY_0) &= m(\angle ABC) - m(\angle Y_0BC) \\
                            &= 2\pi/3 - \Delta -m(\angle Y_0BC) \\
                            &<2\pi/3 - \Delta - \pi/3 +\varepsilon/2 \\
                            &= \pi/3 - \Delta +\varepsilon/2.
        \end{align*}    
        Thus,
        \begin{align*}
            m(\angle AY_0B) &> \pi - \varepsilon/2 - m(\angle ABY_0) \\
                            &> \pi - \varepsilon/2 - (\pi/3 - \Delta +\varepsilon/2) \\
                            &= 2\pi/3 - \varepsilon + \Delta \\
                            &> 2\pi/3.
        \end{align*}    
        Thus, we have $m(\angle AY_0B) > 2\pi/3$. Similarly to the case $Y_1 \neq B$, by the Intermediate Value Theorem, since $m(\angle AY_0B) >\dfrac{2\pi}{3}$ and $m(\angle(AY_2B)) < \dfrac{2\pi}{3}$, then there exists $Y^\ast$ between $Y_0$ and $Y_2$ on $\mathcal{S}$ such that $m(\angle AY^\ast B) = \dfrac{2\pi}{3}$. Thus, we also have $Y^\ast$ is a balanced point.
    \end{itemize}
    Therefore, from both cases, we have proved the existence of a balanced point $Y^\ast$.
    
\end{proof}

\begin{Remark}
    In this proposition, the continuity of the points $Y_X$ as a function of $X \in \gamma_{B,C}$ is crucial. This condition will change for different surfaces as the curvature changes. 
\end{Remark}
    To investigate this condition for the continuity of $Y_X$, we put forward an intermediate proposition.

\begin{Proposition} \label{prop:condition_increase}
    Given any $X$ on $\gamma_{B,C}$ and $Y$ moving on $\gamma_{A,X}$ from $A$ to $X$. If the measure $m(\angle BYC)$ is strictly increasing on a neighborhood near $Y_X$ for every $X$ as a function of $d(A,Y_0)$, then the point $Y_X$ is a continuous function of $X \in \gamma_{B,C}$.
\end{Proposition} 

\begin{figure}[h!]
    \centering
    \begin{subfigure}{.9\textwidth}
  \centering
  \includegraphics[width=.7\linewidth]{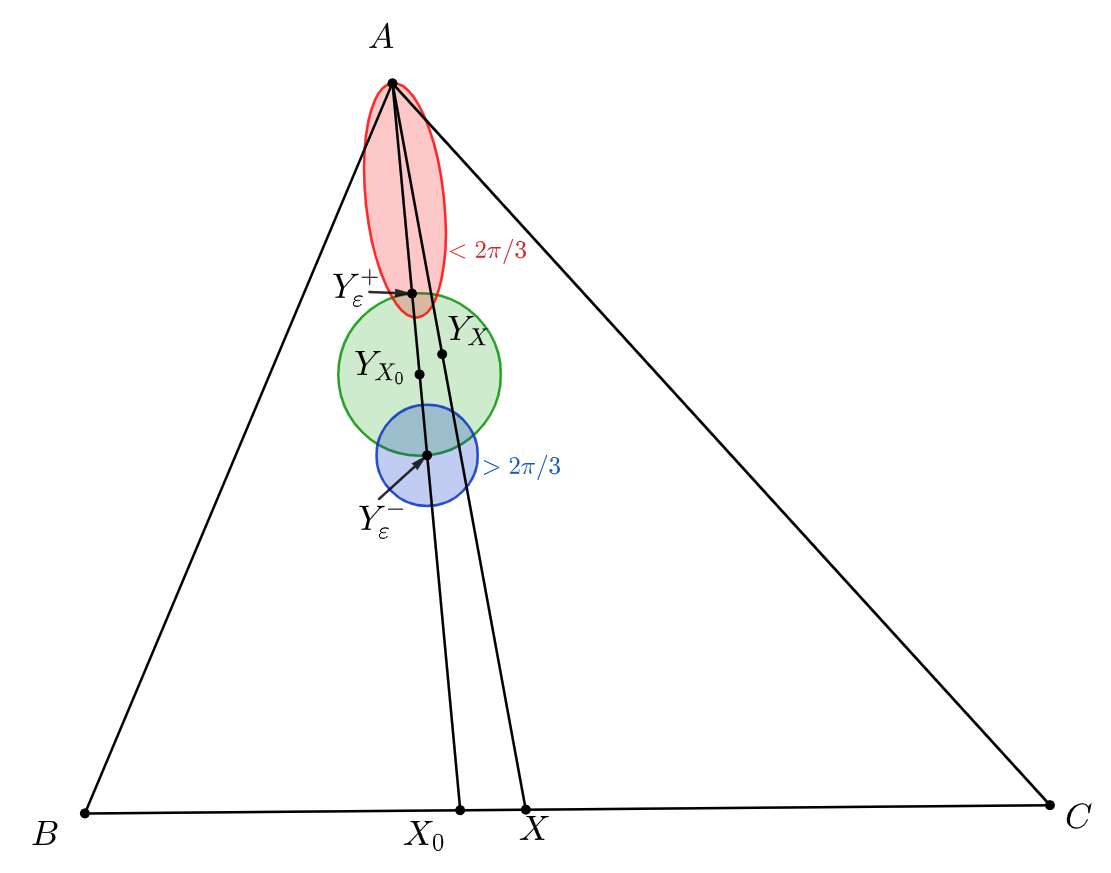}
  \caption{Neighborhoods around $Y_{X_0}$}
  \label{fig:prop_continuity_1}
\end{subfigure}%
\\
\begin{subfigure}{.9\textwidth}
  \centering
  \includegraphics[width=.85\linewidth]{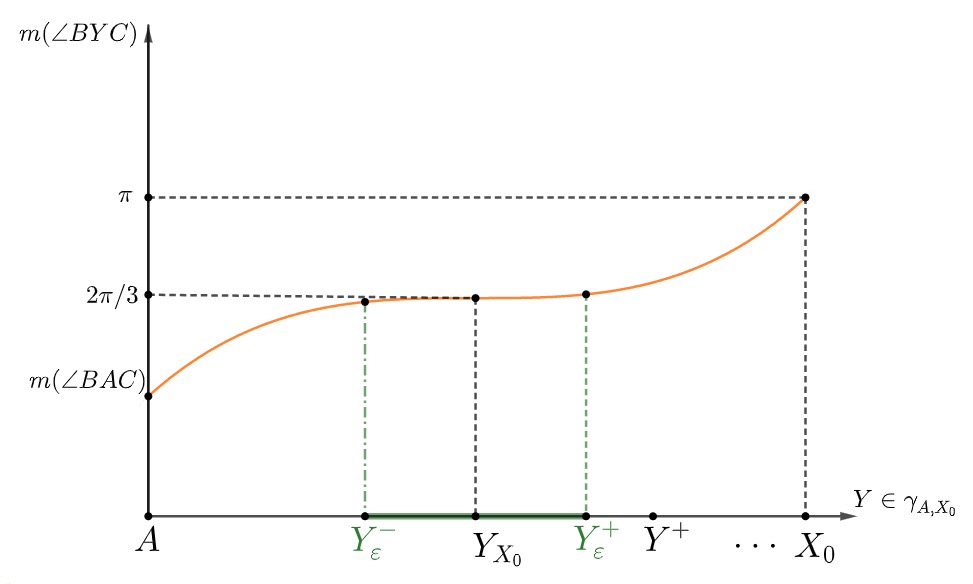}
  \caption{Graph for the angle function of $Y \in \gamma_{A,X_0}$}
  \label{fig:prop_continuity_2}
\end{subfigure}
    \caption{Condition for continuity of $Y_X$}
    \label{fig:prop-con}
\end{figure}
\begin{proof}
    In this proof, we call the function $m(\angle BYC)$ for $Y \in M$ the angle function. Let $X_0$ be a fixed point on $\gamma_{B,C}$ between $B$ and $C$ (exclusive), and define the point $Y_{X_0}$ corresponding to $X_0$. Based on the assumption, there exists a neighborhood around $Y_{X_0}$ such that the measure $m(\angle BYC)$ is strictly increasing as a function of $d(A,Y)$. Thus, there exists a point $Y^+$ between $Y_{X_0}$ and $X_0$ such that for all points $Y$ on $\gamma_{A,X_0}$ between $Y_{X_0}$ and $Y^+$ (exclusive), we have $m(\angle BYC)>\dfrac{2\pi}{3}$.

    Next, let $\varepsilon >0$, and assume that $\varepsilon < d(Y_{X_0},Y^+)$. Let $Y^-_{\varepsilon}$ be a point on $\gamma_{A,X_0}$ between $A$ and $Y_{X_0}$ such that $d(Y^-_{\varepsilon},Y_{X_0}) = \varepsilon$ (Figure \ref{fig:prop_continuity_1}). Similarly define the point $Y^+_{\varepsilon}$ between $Y_{X_0}$ and $Y^+$ such that $d(Y^+_{\varepsilon},Y_{X_0}) = \varepsilon$ (Figure \ref{fig:prop_continuity_1}). To have a better illustration for all points, we can see the graph for the angle function $m(\angle BYC)$ (vertical axis) with respect to the position of $Y$ on $\gamma_{A,X_0}$(horizontal axis) in Figure \ref{fig:prop_continuity_2}.

    We will show that there exists a distance $\delta_\varepsilon$ such that for all points $X$ on $\gamma_{B,C}$ such that $0 < d(X,X_0) < \delta_\varepsilon$, we have the corresponding $Y_X$ is inside the disk
    $$B_{\varepsilon}(Y_{X_0}) = \lbrace Y \in M \,|\, d(Y,Y_{X_0})<\varepsilon \rbrace$$

    First, notice that the angle function $m(\angle BYC)$ as $Y$ moves on the geodesic $\gamma_{A,Y_{\varepsilon}^-}$ is a continuous and closed map. Also, the geodesic $\gamma_{A,Y_{\varepsilon}^-}$ is also a compact set. Thus, the image of the angle function is also compact. By the Extreme Value Theorem, there exists the maximum $\alpha = m(\angle BY^\ast C)$ for $Y^\ast \in \gamma_{A,Y_{\varepsilon}^-}$. Also, $Y_{X_0}$ is the first point that $m(\angle BYC) = \dfrac{2\pi}{3}$ on $\gamma_{A,X_0}$. Then, for each $Y \in \gamma_{A,Y_{\varepsilon}^-}$, we have $ m(\angle BYC) \leq \alpha < \dfrac{2\pi}{3}$. Next, since the surface $M$ is smooth everywhere and the function $m(\angle BYC)$ is also continuous as a function of $Y$, then for each $Y \in \gamma_{A,Y_{\varepsilon}^-}$, there exists a \textcolor{red}{$\delta_Y^->0$} such that the angle function value is \textcolor{red}{less than $\dfrac{2\pi}{3}$} in the neighborhood \textcolor{red}{$B_{\delta_Y^-}(Y)$}. Let $\mathcal{A}^-$ be the family of open disks \textcolor{red}{$B_{\delta_Y^-}(Y)$}, for all $Y \in \gamma_{A,Y_{\varepsilon}^-}$. Then, $\mathcal{A}^-$ is an open covering of $\gamma_{A,Y_{\varepsilon}^-}$, which is a compact set. By the Lebesgue number lemma, $\mathcal{A}^-$ has a Lebesgue number \textcolor{red}{$\delta^-$}. Thus, for every $Y \in \gamma_{A,Y_{\varepsilon}^-}$, the angle function value is less than $\frac{2\pi}{3}$ in the open disk $B_{\frac{\delta^-}{2}}(Y)$ (i.e. the open disk with diameter $\delta^-$). From that, we take the union of all $B_{\frac{\delta^-}{2}}(Y)$ for all $Y \in \gamma_{A, Y_{\varepsilon}^-}$ to make a \textcolor{red}{\textbf{red area} $\mathcal{R}$} in Figure \ref{fig:prop_continuity_1}. From that, we have the area \textcolor{red}{$\mathcal{R}$} where every point there has the angle function value less than $\dfrac{2\pi}{3}$.

    Now, we construct an area around $Y_{\varepsilon}^+$ such that the angle function value is \textcolor{blue}{greater than $\dfrac{2\pi}{3}$}. Since the surface is smooth and the angle function is continuous, and $\widehat{Y_{\varepsilon}^+ B, Y_{\varepsilon}^+ C} > \dfrac{2\pi}{3}$, then there exists \textcolor{blue}{$\delta^+ > 0$} such that for each $Y \in B_{\delta^+}(Y_{\varepsilon}^+)$, we have $m(\angle BYC)>\dfrac{2\pi}{3}$. Then, we have $B_{\delta^+}(Y_{\varepsilon}^+)$ is the \textcolor{blue}{\textbf{blue area} $\mathcal{B}$} in Figure \ref{fig:prop_continuity_1}.

    Then, for all $X$ on $\gamma_{B,C}$ such that the geodesic $\gamma_{A,X}$ is totally inside the union of \textcolor{red}{$\mathcal{R}$}, \textcolor{blue}{$\mathcal{B}$}, and \textcolor{green}{$B_{\varepsilon}(Y_{X_0})$}, we have the point $Y_X$ is in $B_{\varepsilon}(Y_{X_0})$, due to the Intermediate Value Theorem (Figure \ref{fig:prop_continuity_1}). 
    \end{proof}

\begin{Remark}
    This proposition provides a condition for the continuity of $Y_X$ as of function of $X$ on $\gamma_{B,C}$, which is the strict increase of $m(\angle BYC)$ near $Y_X$ as a function of $d(A,Y)$. 
\end{Remark}
In the next subsections, we will investigate this condition for the continuity of $Y_X$ in two cases: on non-positive curvature surfaces and on arbitrary surfaces with curvature bounded from above.

\subsection{Existence of a balanced point on a surface with non-positive curvature}\label{subsection:3.2}

Next, we will show that in the case of a non-positive-curvature surface $M$. The condition from Proposition \ref{prop:condition_increase} is satisfied on any triangle on $M$ with no condition on the length of geodesics. Although this result is ultimately a corollary of Theorem \ref{thm:positive}, we include it because of its simplicity.

\begin{Proposition}\label{thm:non_positive}
    Let $M$ be a non-positive-curvature surface. Then, for any triangle on $M$ such that its three angles measure less than $2\pi/3$, there exists a balanced point.
\end{Proposition}  

\begin{figure}[h!]
    \centering
    \includegraphics[width = 0.45\textwidth]{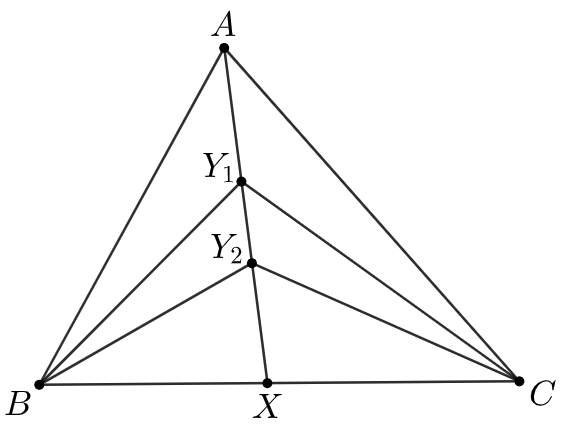}
    \caption{Existence of a balanced point in non-positive-curvature case}
    \label{fig:non-positve}
\end{figure}


\begin{proof}
    Denote $\mathcal{A}$ by the domain of a triangle $ABC$ in $M$. By the Gauss-Bonnet Theorem, we have
    $$\int_{\mathcal{A}} K d \mathcal{A}  = m(\angle BAC) + m(\angle CBA) + m(\angle ACB) - \pi.$$
    For a non-positive curvature surface, we have $K \leq 0$ everywhere. Then, $\int_{\mathcal{A}} K d \mathcal{A} \leq 0$ and 
    $$m(\angle BAC) + m(\angle CBA) + m(\angle ACB) \leq \pi.$$
    Let $X \in \gamma_{B,C}$, and let $Y_1, Y_2  \in \gamma_{A,X}$ such that $Y_1$ is between $A$ and $Y_2$ and $Y_1 \neq Y_2$ (Figure \ref{fig:non-positve}). For triangle $Y_1Y_2B$, 
    $$m(\angle BY_1Y_2) + m(\angle Y_2BY_1) + m(\angle Y_1Y_2B) \leq \pi.$$
    Then,
    \begin{align*}
        m(\angle BY_1X) &<   m(\angle BY_1X) + m(\angle Y_2BY_1) \quad (m(\angle Y_2BY_1) > 0) \\
                                & \leq  \pi - m(\angle Y_1Y_2B) \\
                                &= m(\angle BY_2X)
    \end{align*}
    Similarly, we have $m(\angle XY_1C) < m(\angle XY_2C)$. Thus,
    $$m(\angle BY_1C) = m(\angle BY_1X) + m(\angle XY_1C) <  m(\angle BY_2X)+ m(\angle XY_2C) = m(\angle BY_2C).$$
    From that, we have the measure $m(\angle BYC)$ increases when $Y$ moves from $A$ to $X$ on $\gamma_{A,X}$, for all $X \in \gamma_{B,C}$. Then, from Proposition \ref{prop:condition_increase}, there exists a balanced point in $\triangle ABC$.
\end{proof}

\subsection{Existence of a balanced point on a general surface} 
\label{subsection:3.3}
In the case of a non-positive curvature surface, the Gauss-Bonnet Theorem implies the sum of three angles in a triangle is less than or equal to $\pi$. However, in the case of arbitrary curvature, the sum of three angles in a triangle can be higher than $\pi$, which leads to complicated behavior of the angle $\angle BYC$ as $Y$ moves on $\gamma_{A,X}$. In this subsection, we first investigate the condition from Proposition \ref{prop:condition_increase} on a general 2-D sphere with radius $R$. Then, we will extend that condition for a general surface with curvature bounded above by $1/R^2$, which is the curvature of a sphere with radius $R$. The idea is to use Jacobi fields.

We first come to the proposition of putting a condition for the increase of angles along the geodesic as in Proposition \ref{prop:condition_increase}.

\begin{Proposition} \label{prop:90_increase}
    Let $X$ be a point on $\gamma_{B,C}$. Let $Y_1$ be a point on geodesic $\gamma_{A,X}$ between $A$ and $X$. Let $Y_2$ be a point on $\gamma_{A,X}$ between $Y_1$ and $X$, and let $T$ be a point on geodesic $\gamma_{B,Y_2}$ such that $\gamma_{Y_1,T}$ is orthogonal to $\gamma_{B,Y_1}$. Let $\varepsilon > 0$. Assume that $Y_2$ and $T$ are in a small neighborhood of $Y_1$ such that 
    \begin{align}
        \label{triangle-epsilon}
            \pi - \varepsilon < m(\angle TY_1Y_2) + m(\angle TY_2Y_1) + m(\angle Y_1TY_2) < \pi + \varepsilon.
    \end{align}
    Then, we have a condition:
    $$\text{If $m(\angle Y_1TB) < \pi/2 - \varepsilon$, then $m(\angle BY_2X) > m(\angle BY_1X).$}$$
    
\end{Proposition}

\begin{figure}[h!]
    \centering
    \begin{subfigure}{.33\textwidth}
  \centering
  \includegraphics[width=.9\linewidth]{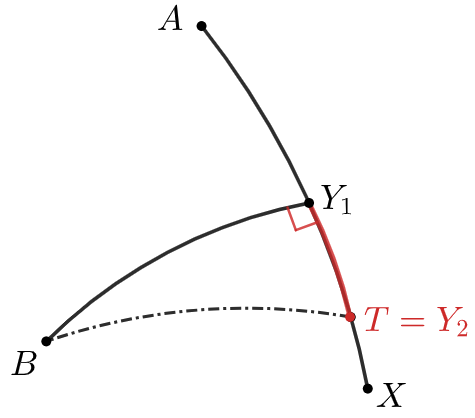}
  \caption{Case 1: $m(\angle BY_1X) = \pi/2$}
  \label{fig:increase_angle_e90}
\end{subfigure}
    \begin{subfigure}{.33\textwidth}
  \centering
  \includegraphics[width=.9\linewidth]{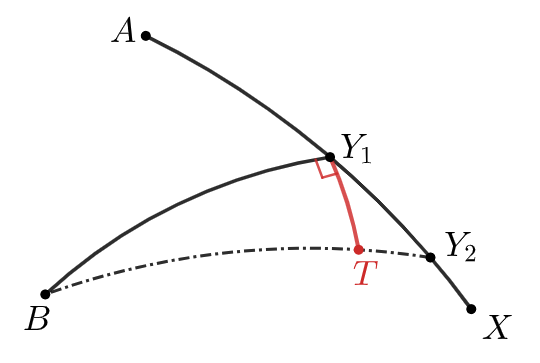}
  \caption{Case 2: $m(\angle BY_1X) > \pi/2$}
  \label{fig:increase_angle_h90}
\end{subfigure}%
\begin{subfigure}{.33\textwidth}
  \centering
  \includegraphics[width=.9\linewidth]{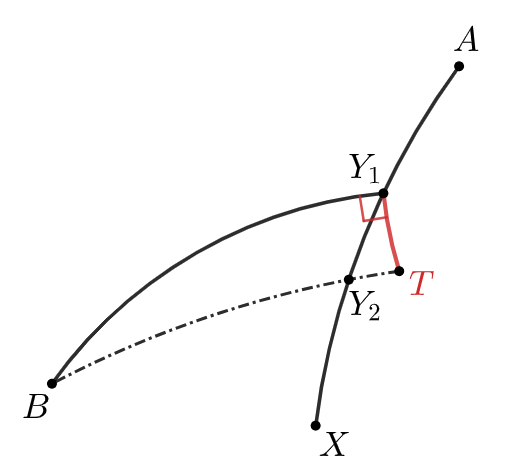}
  \caption{Case 2: $m(\angle BY_1X) < \pi/2$}
  \label{fig:increase_angle_l90}
\end{subfigure} 
    \caption{The increase of angle}
    \label{fig:increase_angle_lemma}
\end{figure}

\begin{proof}
     Assume that the angle $\angle Y_1TB$ measures less than $\pi/2$. There are three main cases:
    \begin{itemize}
        \item \textbf{Case 1:} $m(\angle BY_1X) = \pi/2$. This case is clear because $T \equiv Y_2$ and $$m(\angle BY_2X) = \pi - m(\angle BTY_1) > \pi - (\pi/2 - \varepsilon) = \pi/2 + \varepsilon > \pi/2 = m(\angle BY_1X).$$
        \item \textbf{Case 2:} $m(\angle BY_1X) > \pi/2$. Since $\angle BY_1Y_2$ measures larger than $\angle BY_1T$, then the geodesic $\gamma_{Y_1,T}$ is between geodesics $\gamma_{Y_1,B}$ and $\gamma_{Y_1Y_2}$. 
        Thus,
        \begin{align*}
            m(\angle BY_2X) &= \pi - m(\angle BY_2Y_1) \\
                            &> \pi + m(\angle TY_1Y_2) + m(\angle Y_2TY_1) - \pi - \varepsilon \quad (\text{From (\ref{triangle-epsilon})}) \\
                            &= m(\angle TY_1Y_2) + (\pi - m(\angle Y_1TB)) -\varepsilon\\
                            &> m(\angle TY_1Y_2) + \pi - (\pi/2 - \varepsilon)  - \varepsilon\\
                            &= m(\angle TY_1Y_2) + \pi/2 \\
                            &= m(\angle TY_1Y_2) + m(\angle BY_1T) \\
                            &= m(BY_1X).
        \end{align*}    
        \item \textbf{Case 3:} $m(\angle BY_1X) < \pi/2$. We have the geodesic $\gamma_{Y_1,Y_2}$ is between geodesic $\gamma_{Y_1,B}$ and geodesic $\gamma_{Y_1,T}$.
        Thus,
        \begin{align*}
            m(\angle BY_2X) &= m(\angle Y_1Y_2T) \\
                            &> \pi - \varepsilon - m(\angle TY_1Y_2) - m(\angle Y_1TY_2) \quad \text{(From (\ref{triangle-epsilon}))} \\
                            &> \pi - \varepsilon - m(\angle TY_1Y_2) - (\pi/2 - \varepsilon)\\
                            &= \pi/2 - m(\angle TY_1Y_2)\\
                            &= m(\angle BY_1T) - m(\angle TY_1Y_2) \\
                            &= m(\angle BY_1X).
        \end{align*}
    \end{itemize}
    Therefore, in all three cases, we have $m(\angle BY_2X) > m(\angle BY_1X)$.
\end{proof}

\begin{Remark}
    Since the proposition is true for all $\varepsilon>0$, we can obtain the condition:
    $$\text{If $m(\angle Y_1TB) < \pi/2$ and $T$ is sufficiently close to $Y_1$, then $m(\angle BY_2X) > m(\angle BY_1X).$}$$
\end{Remark}
This proposition allows us to work with right angles instead of general angles in terms of proving the increase of the measures of $\angle BYX$ and $\angle CYX$ at $Y_X$ in Proposition \ref{prop:condition_increase}. From that, we can take advantage of the Jacobi Fields along geodesic $\gamma_{B,Y}$ that are perpendicular to that geodesic. Next, we discuss how the change of the magnitude of Jacobi fields at the point $Y_1$ relates to the increase in the measure of $\angle BY_1X$.
\begin{Proposition} \label{prop:derivative_jacobi_positive}
    Let $\gamma(\theta,t)$ be a variation of geodesics, based on the Gauss lemma (Lemma 3.5 \cite{do1992riemannian}), for some $\epsilon>0$, given by
    $$\gamma(\theta,t) = \exp_{B}tv(\theta), \quad 0 \leq t \leq 1,\quad -\epsilon \leq \theta \leq \epsilon,$$
    such that  $v(\theta)$ is a curve in $T_B M $, $\gamma(0,t) = \gamma_{B,Y_1}$, $\gamma(0,0)=B$, and $\gamma(0,t_0) = Y_1$, for $t_0 \leq 1$. Assume that $\gamma(\theta,0) = B$, which means the geodesic is fixed at a boundary point $B$ when $\theta$ varies. Let $J(t)$ be a Jacobi field along $\gamma(0,t)$ such that $J(t) = \dfrac{d\gamma}{d\theta}(0,t)$, $J(0)=0$, $\lVert J'(0) \rVert > 0$, and $\langle J(t),\gamma_t(0,t) \rangle = 0$, for $t \in [0,t_0]$. Assume that $T = \gamma(\theta_0,t_0)$, for some $\theta_0 \in (0, \varepsilon]$, and $v(\theta)$ varies such that $\gamma(\theta,t_0)$ is a constant-speed geodesic with parameter $\theta$. Let $f: [0,\theta_0] \to M$ such that $f(\theta)=\gamma(\theta,t_0)$ for $\theta \in [0,\theta_0]$. We have a new condition:
    $$\text{If $(\lVert J\rVert^2)'(t_0) > 0$, then $m(\angle Y_1TB) < \pi/2$.}$$
\end{Proposition}
\begin{figure}[h!]
    \centering
    \includegraphics[width=0.8\linewidth]{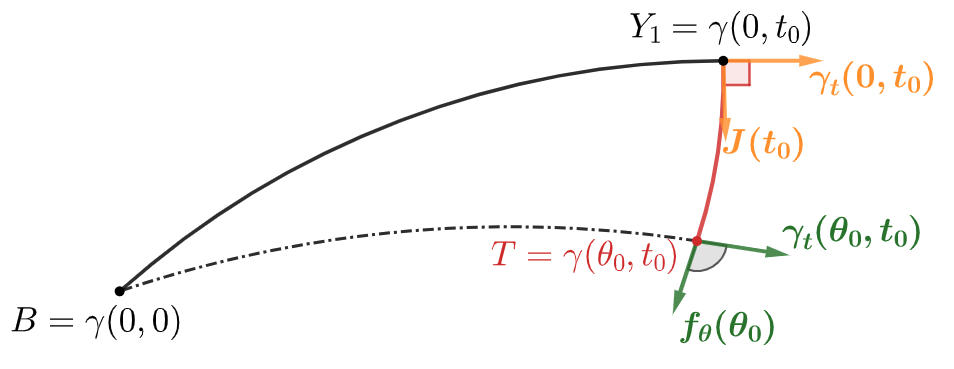}
    \caption{Jacobi Field and angles}
    \label{fig:jacobi_lemma}
\end{figure}

\begin{proof}
Assume that $(\lVert J \rVert^2)'(t_0) > 0$. Since $\gamma_{Y_1,T}$ is perpendicular to $\gamma_{B,Y_1}$, then $f_\theta(0) \perp \gamma_t(0,t_0)$. Since $J(t_0) \perp \gamma_t(0,t_0)$ and $J(t_0) \in T_{Y_1}M$, then $f_\theta(0) = c J(t_0)$, for a non-zero scalar $c$. By construction, $c>0$.

Next, we have the measure of $\angle Y_1TB$ is equal to the measure of the angle between $\gamma_t(\theta_0,t_0)$ and $f_{\theta}(\theta_0)$ (two tangent vectors at $T$). Then, we have
\begin{align*}
    m(\angle Y_1TB) < \pi/2 &\Leftrightarrow \text{The angle between $\gamma_t(\theta_0,t_0)$ and $f_{\theta}(\theta_0)$ measures less than $\pi/2$} \\
    &\Leftrightarrow \langle \gamma_t(\theta_0,t_0),f_{\theta}(\theta_0) \rangle > 0.
\end{align*}
Now, consider the function $g(\theta) = \langle \gamma_t(\theta,t_0),f_{\theta}(\theta) \rangle$. We have $g(0) = \langle \gamma_t(0,t_0),f_{\theta}(0) \rangle = 0$. To prove that $g(\theta_0) = \langle \gamma_t(\theta_0,t_0),f_{\theta}(\theta_0) \rangle > 0$, we will show that $\left(\dfrac{dg}{d\theta} \right)_{\theta=0}> 0$.

Indeed, we have
\begin{align*}
  \left(\dfrac{dg}{d\theta} \right)_{\theta=0} &= \left(\dfrac{d}{d\theta} \langle \gamma_t(\theta,t_0),f_{\theta}(\theta) \rangle \right)_{\theta=0}   \\
  &=\left \langle \left( \frac{D}{\partial \theta} \gamma_t \right)_{(\theta,t)=(0,t_0)} , f_{\theta}(0) \right \rangle + \left \langle \gamma_t(0,t_0),\left( \frac{D}{d\theta} f_\theta(\theta) \right)_{\theta=0}  \right \rangle
\end{align*}
We have $f(\theta)$ parametrizes a geodesic, then $\frac{D}{d\theta} f_\theta(\theta) = 0$, for $\theta \in [0,\theta_0]$. Now we want to show that
$$\frac{D}{\partial \theta} \gamma_t = \frac{D}{\partial t} \gamma_\theta$$
Indeed, let $\partial_{\theta}$ and $\partial_t$ be two coordinate vector fields of the surface near $\gamma(0,t_0)$. Let $\nabla$ be the Levi-Civita connection on $M$. Then, 
$$\frac{D}{\partial \theta} \partial_t = \nabla_{\partial_\theta} \partial_t \quad \text{and} \quad \frac{D}{\partial t} \partial_\theta = \nabla_{\partial_t}\partial_\theta,$$
with $\partial_t = \gamma_t(\theta,t)$ and $ \partial_\theta = \gamma_\theta(\theta, t)$, which is the same as $f_\theta(\theta)$ when $t=t_0$. Since $\gamma$ is smooth, then
$$[\partial_\theta,\partial_t] = \partial_\theta \partial_t - \partial_t \partial_\theta = 0.$$
Thus, we have
$$\frac{D}{\partial \theta} \partial_t - \frac{D}{\partial t} \partial_\theta = \nabla_{\partial_\theta} \partial_t - \nabla_{\partial_t}\partial_\theta 
    = [\partial_\theta,\partial_t] 
    =0.$$
Therefore, 
$$\frac{D}{\partial \theta} \gamma_t = \frac{D}{\partial t} \gamma_\theta.$$
Next, we have
\begin{align*}
    \left \langle \left( \frac{D}{\partial \theta} \gamma_t \right)_{(\theta,t)=(0,t_0)} , f_{\theta}(0) \right \rangle &= \left \langle \left( \frac{D}{\partial t} \gamma_\theta \right)_{(\theta,t)=(0,t_0)} , f_{\theta}(0) \right \rangle \\
    &= \left \langle \left( \frac{D}{dt} \gamma_\theta(0,t)\right)_{t=t_0} , f_{\theta}(0) \right \rangle \\
    &= \left \langle \left( \frac{D}{dt} J(t)\right)_{t=t_0} , cJ(t_0) \right \rangle \\
    &= c\langle J',J \rangle(t_0) \\
    &= \frac{c}{2} (\langle J,J\rangle)'(t_0) \\
    &= \frac{c}{2} (\lVert J \rVert^2 )'(t_0) \\
    &>0.
\end{align*}
Therefore, 
$$ \left(\dfrac{dg}{d\theta} \right)_{\theta=0} = \underbrace{\left \langle \left( \frac{D}{\partial \theta} \gamma_t \right)_{(\theta,t)=(0,t_0)} , f_{\theta}(0) \right \rangle}_{>0} + \left \langle \gamma_t(0,t_0),\underbrace{\left( \frac{D}{d\theta} f_\theta(\theta) \right)_{\theta=0}}_{=0} \right \rangle > 0.$$
\end{proof}

By combining this proposition with Proposition \ref{prop:90_increase}, we just need to investigate when $(\lVert J \rVert^2)'(t_0) > 0$ to achieve the increasing-angle condition in Proposition \ref{prop:condition_increase}. This is also the main condition we will use for the next results in the case of surfaces with curvatures bounded from above. 
\\

Now, we start with the constant positive curvature case by investigating $(\lVert J(t) \rVert^2)'$ on a round sphere with radius $R$. The below proposition is shown in \textbf{Example 2.3, Chapter 5} of \cite{do1992riemannian}.

\begin{Proposition} \label{prop:sphere_jacobi_field}
    Let $\gamma$ be a normalized geodesic on a round sphere with radius $R$. Let $J$ be a Jacobi field along $\gamma$ that is orthogonal to $\gamma'$. Let $w(t)$ be a parallel field along $\gamma$ with $\langle \gamma'(t), w(t)\rangle = 0$ and $\lVert w(t) \rVert = 1$. Then,
    $$
    J(t) = R\sin \left( \frac{t}{R} \right) w(t).
    $$
\end{Proposition}

\begin{Corollary} \label{col: sphere_jacobi}
    If $\gamma$ is a normalized geodesic, i.e. $\lVert \gamma' \rVert = 1$, then $(\lVert J(t) \rVert^2)' >0
    $ for all $t \in (0,R\pi/2)$. Therefore, if the maximum geodesic distance of two points in the domain of the triangle $ABC$ on a sphere with radius $R$ is less than $R\pi/2$, it will satisfy the condition in Proposition \ref{prop:condition_increase}.
\end{Corollary}

\begin{proof}
    From Proposition \ref{prop:sphere_jacobi_field}, we have
    $$\lVert J(t) \rVert^2 = \langle J(t), J(t) \rangle = R^2 \sin^2(t/R).$$
    Then,
    $$(\lVert J(t) \rVert^2)' = (R^2 \sin^2(t/R))' = 2R \sin (t/R)\cos(t/R).$$
    Thus, for all $t \in (0,R\pi/2)$, we have
    $$(\lVert J(t) \rVert^2)' = 2R \sin (t/R)\cos(t/R) > 0.$$
\end{proof}

From this, we can describe the condition for the angle increasing through the Jacobi Field on 2-D spheres. Next, we combine Corollary \ref{col: sphere_jacobi} with Proposition \ref{prop:condition_increase} to prove the existence of a balanced point on a round sphere.

\begin{Proposition}
    \label{thm:sphere}
    Given triangle $ABC$ on a round sphere $M$ with Gaussian curvature $1/R^2$ such that its three angles measure less than $2\pi/3$. If the maximum geodesic distance of two points in the domain of the triangle $ABC$ is less than $R\pi/2$, then there exists a balanced point.
\end{Proposition}

\begin{proof}
Let $X$ be a point on $\gamma_{B,C}$ and $Y$ be a point on $\gamma_{A,X}$. We have the maximum lengths of $\gamma_{B,Y}$ and $\gamma_{C,Y}$ are both less than $R\pi/2$. By Corollary \ref{col: sphere_jacobi}, we have $m(\angle BYX)$ and $m(\angle CYX)$ are increasing as $Y$ moves from $A$ to $X$ (Figure \ref{fig:theorem_sphere}). From that, by Proposition \ref{prop:condition_increase}, there exists a balanced point in the triangle $ABC$.  
\end{proof}

\begin{figure}[h!]
    \centering
    \includegraphics[width=0.35\linewidth]{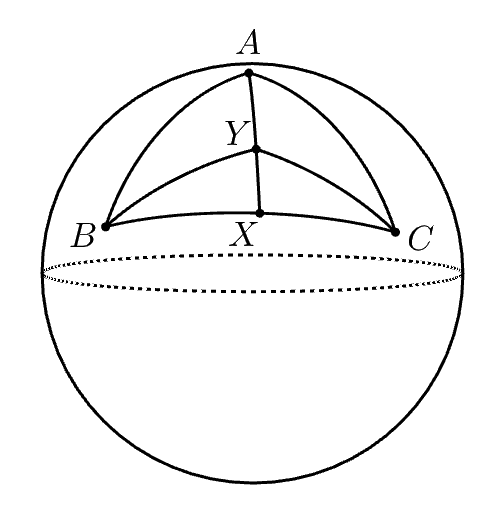}
    \caption{Triangle $ABC$ on a sphere}
    \label{fig:theorem_sphere}
\end{figure}

Now, we generalize the condition to general surfaces with Gaussian curvature $K$ bounded above by the curvature of a sphere as Theorem \ref{thm:positive}

\begin{proof}[\textbf{Proof of Theorem \ref{thm:positive}}]
Let $X$ be a point on $\gamma_{B,C}$ and $Y$ be a point on $\gamma_{A,X}$. Let $\gamma(t):[0,a] \to M$ parameterize $\gamma_{B,X}$ such that $\gamma(0) = B, \gamma(a) = X$, and $\lVert \gamma'(t) \rVert = 1$, for $t \in [0,a]$. Let $J(t)$ be a Jacobi field along $\gamma(t)$ such that $J(0)=0$, $\lVert J'(0) \rVert > 0$, and $\langle J(t), \gamma'(t) \rangle = 0$, for all $t \in [0,a]$. We need to prove $(\lVert J(t) \rVert^2)'>0$ so that we can use Proposition \ref{prop:derivative_jacobi_positive} and Proposition \ref{prop:90_increase} to achieve the increase-angle condition in Proposition \ref{prop:condition_increase}.

Let $\widetilde{M}$ be a sphere with radius $R$. Then, $M$ has the Gaussian curvature $\widetilde{K} = \frac{1}{R^2}$. Let $\widetilde{\gamma}:[0,a] \to \widetilde{M}$ be a geodesic with unit velocity (i.e. $\lVert \widetilde{\gamma}'(t) \rVert = 1 = \lVert \gamma'(t) \rVert$). Then, $\widetilde{\gamma}(t)$ has the same length as $\gamma(t)$ as $t \in [0,a]$ that is less than $R\pi/2$. Let $\Tilde{J}(t)$ be a Jacobi field along $\widetilde{\gamma}$ satisfying these conditions:
\begin{itemize}
    \item $\Tilde{J}(0) = J(0) = 0$
    \item $\langle \Tilde{J}(t),\widetilde{\gamma}'(t) \rangle =  \langle J(t), \gamma'(t) \rangle = 0$, for all $t \in [0,a]$
    \item $\lVert \Tilde{J}'(0) \rVert = \lVert J'(0) \rVert > 0$.
\end{itemize}
Since the length of $\widetilde{\gamma}$ is less than $R\pi/2$, then by Corollary \ref{col: sphere_jacobi}, we have $(\lVert \Tilde{J}(t) \rVert^2)' > 0$, for all $t \in [0,a]$. Our next goal is proving that  $(\lVert J(t) \rVert^2)' \geq (\lVert \Tilde{J}(t) \rVert^2)'$, for all $t \in [0,a]$.

\begin{figure}[h!]
    \centering
    \includegraphics[width=0.8\linewidth]{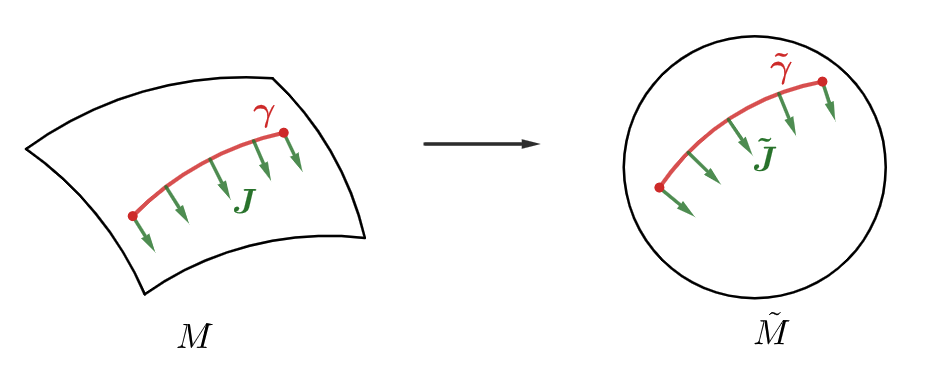}
    \caption{Relating Jacobi fields in $M$ and $\widetilde{M}$}
    \label{fig:theorem_positive}
\end{figure}

Indeed, let $v(t) := \lVert J(t) \rVert^2$ and $\Tilde{v}(t) := \lVert \Tilde{J}(t) \rVert^2$. We have the length of $\widetilde{\gamma}(t)$ for $t \in (0,a]$ is less than $R\pi/2$. Then $\widetilde{\gamma}(t)$ has no conjugate point when $t \in (0,a]$. Thus, $\Tilde{v}(t) = \lVert \Tilde{J}'(t) \rVert^2 > 0$, for all $t \in [0,a]$. From the L'Hospital's rule,
\begin{align*}
    \lim_{t \to 0^+} \frac{v(t)}{\Tilde{v}(t)} &= \lim_{t \to 0^+} \frac{v'(t)}{\Tilde{v}'(t)} = \lim_{t \to 0^+} \frac{\langle J(t),J(t) \rangle'}{ \langle \Tilde{J}(t), \Tilde{J}(t) \rangle '} = \lim_{t \to 0^+} \frac{2\langle J'(t),J(t) \rangle}{ 2\langle \Tilde{J}'(t), \Tilde{J}(t) \rangle} \\
    &= \lim_{t \to 0^+} \frac{\langle J''(t), J(t) \rangle + \langle J'(t),J'(t) \rangle}{\langle \Tilde{J}''(t), \Tilde{J} \rangle + \langle \Tilde{J}'(t), \Tilde{J}'(t) \rangle} \\
    &= \frac{\langle J''(0), J(0) \rangle + \langle J'(0),J'(0) \rangle}{\langle \Tilde{J}''(0), \Tilde{J}(0) \rangle + \langle \Tilde{J}'(0), \Tilde{J}'(0) \rangle}\\
    &= \frac{\lVert J'(0) \rVert^2}{\lVert \Tilde{J}'(0) \rVert^2} \quad \quad (\text{Since } J(0) = \Tilde{J}(0) = 0)\\
    &= 1.
\end{align*}
From that,
$$\lim_{t \to 0^+} \frac{v'(t)}{\Tilde{v}'(t)} = \lim_{t \to 0^+} \frac{(\lVert J(t) \rVert^2)'}{(\lVert \Tilde{J}(t) \rVert^2)'} = 1.$$
Therefore, to prove that $v'(t) \geq \Tilde{v}'(t)$, we need to show that $\dfrac{d}{dt}\left( \dfrac{v'(t)}{\Tilde{v}'(t)}\right) \geq 0$, or equivalently,
\begin{align}\label{v_condition}
    v'(t)\Tilde{v}(t) \geq v(t)\Tilde{v}'(t), \quad \text{for all }t \in [0,a]. 
\end{align}

Now, fix $t_0 \in (0,a]$. If $v(t_0) = 0$, then $J(t_0)=0$ and
$$v'(t_0) = 2 \langle J'(t_0), J(t_0) \rangle = 0.$$
Then, both sides of (\ref{v_condition}) are 0.

Suppose that $v(t_0) \neq 0$. Then, define
$$U(t) = \frac{1}{\sqrt{v(t_0)}}J(t), \quad \Tilde{U}(t) = \frac{1}{\sqrt{\Tilde{v}(t_0)}}\Tilde{J}(t).$$
Then,
\begin{align*}
    \dfrac{v'(t_0)}{v(t_0)} &= \dfrac{2\langle J'(t_0),J(t_0) \rangle}{\langle J(t_0),J(t_0) \rangle} = 2\langle U'(t_0),U(t_0) \rangle = 2\langle U',U \rangle(t_0) \\
    &= 2\int_0^{t_0} \langle U',U \rangle' dt = 2\int_0^{t_0} (\langle U',U' \rangle + \langle U'',U \rangle) dt \\
    &= 2\int_0^{t_0} (\langle U',U' \rangle - \langle KU,U \rangle) dt \\
    &= 2\int_0^{t_0} (\langle U',U' \rangle - K\langle U,U \rangle) dt \\
    &= 2I_{t_0}(U,U).
\end{align*}
Similarly, we have
$$\dfrac{\Tilde{v}'(t_0)}{\Tilde{v}(t_0)} = 2I_{t_0}(\Tilde{U},\Tilde{U}).$$
Therefore, to show $\dfrac{v'(t_0)}{v(t_0)} \geq \dfrac{\Tilde{v}'(t_0)}{\Tilde{v}(t_0)}$, we need to prove that 
$$I_{t_0}(U,U) \geq I_{t_0}(\Tilde{U},\Tilde{U}).$$

Indeed, let $\lbrace e_1(t),e_2(t) \rbrace$ and $\lbrace \Tilde{e}_1(t), \Tilde{e}_2(t) \rbrace$ be parallel orthonormal bases along $\gamma(t)$ and $\widetilde{\gamma}(t)$, respectively, such that
$$e_1(t) = \gamma'(t)/\lVert \gamma'(t) \rVert, \quad e_2(t) = U(t_0),$$
$$\Tilde{e}_1(t) = \widetilde{\gamma}'(t)/\lVert \widetilde{\gamma}'(t) \rVert, \quad \Tilde{e}_2(t) = \Tilde{U}(t_0).$$

Then, write $U(t) = g_1(t)e_1(t) + g_2(t)e_2(t)$ along $\gamma$. Also, define a map $$\phi: T_{\gamma(t)}M \to T_{\widetilde{\gamma}(t)}M$$ 
such that
$$(\phi U)(t) = g_1(t)\Tilde{e}_1(t)+g_2(t)\Tilde{e}_2(t).$$
Then, we have these properties:
    \item \begin{align*}
        \langle \phi U , \phi U \rangle &= \langle g_1(t)\Tilde{e}_1(t)+g_2(t)\Tilde{e}_2(t), g_1(t)\Tilde{e}_1(t)+g_2(t)\Tilde{e}_2(t) \rangle \\
        &= g_1^2(t) + g_2^2(t)\\
        &= \langle g_1(t)e_1(t) + g_2(t)e_2(t), g_1(t)e_1(t) + g_2(t)e_2(t) \rangle \\
        &= \langle U,U\rangle
    \end{align*}
and $$(\phi U)' = \left( g_1(t)\Tilde{e}_1(t)+g_2(t)\Tilde{e}_2(t) \right)' =  g_1'(t)\Tilde{e}_1(t)+g_2'(t)\Tilde{e}_2(t) = \phi(U').$$
From that, we also have
$$ \langle (\phi U)', (\phi U)' \rangle = \langle \phi(U'),\phi(U') \rangle = \langle U',U' \rangle.$$
Since $K \leq \widetilde{K}$, we have
\begin{align*}
    I_{t_0}(U,U) &= \int_0^{t_0} (\langle U',U' \rangle - K\langle U,U \rangle) dt\\
                 &\geq \int_0^{t_0} (\langle U',U' \rangle - \widetilde{K}\langle U,U \rangle) dt  \\
                 &= \int_0^{t_0} (\langle (\phi U)',(\phi U)' \rangle - \widetilde{K} \langle \phi U, \phi U \rangle) dt \\
                 &= I_{t_0}(\phi U, \phi U).
\end{align*}
Thus, 
$$I_{t_0}(U,U) \geq I_{t_0}(\phi U, \phi U).$$

On the other hand, we have $\Tilde{U}$ and $\phi U$ are two vector fields along $\widetilde{\gamma}$ that satisfy all hypothesis conditions in Lemma \ref{lemma:index}. Also, $\Tilde{U}$ is a Jacobi field along $\widetilde{\gamma}$. By Lemma \ref{lemma:index}, we have
$$I_{t_0}(\phi U, \phi U) \geq I_{t_0}( \Tilde{U}, \Tilde{U}).$$
Therefore,
$$I_{t_0}(U,U) \geq I_{t_0}(\phi U, \phi U) \geq I_{t_0}( \Tilde{U}, \Tilde{U}).$$
Hence, we have $(\lVert J(t) \rVert^2)' \geq (\lVert \Tilde{J}(t) \rVert^2)' > 0$, for all $t \in [0,a]$. Thus, the increase-angle condition in Proposition \ref{prop:condition_increase} is satisfied. Therefore, there exists a balanced point inside triangle $ABC$.
\end{proof}

\begin{Remark}
    This proof is mostly covered in the proof of The Rauch Comparion Theorem \cite{do1992riemannian}. In this problem, we just consider the case of a 2-D surface instead of an $n$-dimensional manifold.
\end{Remark}

From Proposition \ref{thm:sphere} and Theorem \ref{thm:positive}, the bounded maximum geodesic distance in the region of triangle $ABC$ is a crucial condition. That fact raises a question: \emph{Is there any triangle $ABC$ that has the maximum side length larger or equal to $R\pi/2$ that does not have a balanced point?}

Here, we construct a simple example to show the answer ``yes" to the above question.
\begin{figure}[h!]
    \centering
    \includegraphics[width=0.3\linewidth]{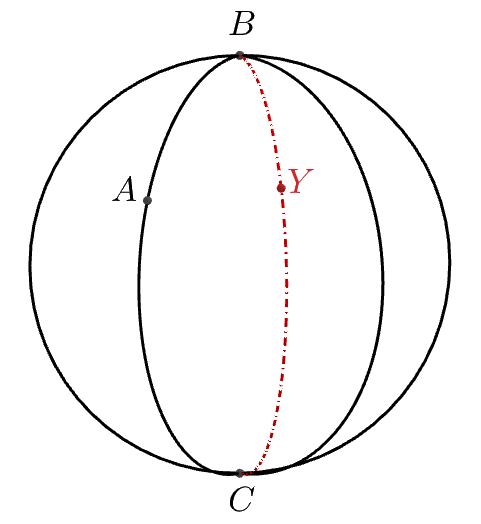}
    \caption{Example of a triangle having no balanced point}
    \label{fig:example_positive}
\end{figure}
\begin{Example}
    Let $M$ be a 2-D sphere with radius $R$. Let $A$ and $C$ be two distinct points on the sphere, and let $B$ be the antipodal point of $A$ (Figure \ref{fig:example_positive}). Then, the triangle $ABC$ does not have a balanced vertex $Y$ inside the triangle because for all points $Y$ inside $\triangle ABC$, the measure of $\angle BYA$ is always $\pi$, which cannot be $2\pi/3$.
\end{Example}

In addition, we show an example where the side lengths are greater than $R\pi/2$, and where there still exists a balanced point. 

\begin{figure}[h!]
    \centering
    \includegraphics[width=0.32\linewidth]{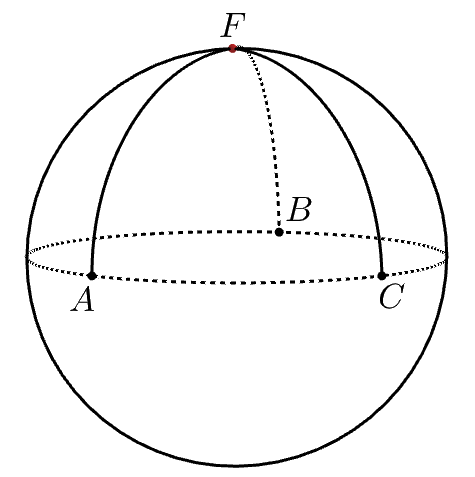}
    \caption{Example of a triangle with side lengths larger than $R\pi/2$ which has a balanced point.}
    \label{fig:example-higher-angle}
\end{figure}

\begin{Example} Let $M$ be a 2-D sphere with radius $R$. On the horizontal equator, take 3 points $A,B,C$ such that triangle $ABC$ is equilateral. Also, take $F$ as one of two poles of the sphere (Figure \ref{fig:example-higher-angle}). We have each side length of triangle $ABC$ is $\frac{2R\pi}{3} > \frac{R\pi}{2}$. Also, we have 
$$m(\angle AFB) = m(\angle BFC) = m(\angle CFA) = 2\pi/3.$$
That means $F$ is a balanced point of triangle $ABC$.
\end{Example}

\bibliographystyle{plain}
\bibliography{ref}

\end{document}